\documentclass{amsart}
\usepackage{amssymb}

\usepackage[british,UKenglish,USenglish,american]{babel}
\usepackage{graphicx}         
\usepackage{fancyhdr}
\usepackage{rotating}
\usepackage{amsmath}
\usepackage{amssymb}
\usepackage{amsthm}
\usepackage{tikz}
\usepackage{url}
\usepackage{enumerate}
\usepackage{mathtools}
\usepackage{setspace}
\usepackage{color}
\usepackage[normalem]{ulem}
\usetikzlibrary{positioning}

\newtheorem*{notation*}{Notation}

\newtheorem{innercustomthm}{Theorem}
\newenvironment{customthm}[1]
  {\renewcommand\theinnercustomthm{#1}\innercustomthm}
  {\endinnercustomthm}
  
\usepackage{environ}
\newcounter{quote}

\NewEnviron{myquotenumber}{\vspace{3ex}\par
\refstepcounter{quote}%
\hfill\parbox{\dimexpr \textwidth-2cm}
{\centering\small\textit{\BODY}}
\hfill\llap{(\thequote)}\vspace{2ex}\par}

\def\Ind#1#2{#1\setbox0=\hbox{$#1x$}\kern\wd0\hbox to 0pt{\hss$#1\mid$\hss}
\lower.9\ht0\hbox to 0pt{\hss$#1\smile$\hss}\kern\wd0}

\def\notind#1#2{#1\setbox0=\hbox{$#1x$}\kern\wd0
\hbox to 0pt{\mathchardef\nn=12854\hss$#1\nn$\kern1.4\wd0\hss}
\hbox to 0pt{\hss$#1\mid$\hss}\lower.9\ht0 \hbox to 0pt{\hss$#1\smile$\hss}\kern\wd0}

\newtheorem{thm}{Theorem}[section]
\newtheorem{cor}[thm]{Corollary}
\newtheorem{prop}[thm]{Proposition}

\newtheorem{lem}[thm]{Lemma}
\newtheorem{conj}[thm]{Conjecture}

\theoremstyle{definition}
\newtheorem{defn}[thm]{Definition}

\newtheorem{fact}[thm]{Fact}

\theoremstyle{remark}

\newtheorem*{theorem*}{Theorem}

\title[The Borovik--Cherlin conjecture holds in ACF]{The Borovik--Cherlin conjecture holds in ACF}

\author[U. Karhum\"{a}ki]{Ulla Karhum\"{a}ki$^\dagger$}
\address{Universit\'{e} Claude Bernard Lyon 1;  Institut Camille Jordan}
\email{karhumaki@math.univ-lyon1.fr}

\author[N. Ramsey]{Nicholas Ramsey$^{\ddagger}$}
\address{Department of Mathematics \\
University of Notre Dame\\
 USA}
\email{sramsey5@nd.edu}

\thanks{$^\dagger$ Karhum\"{a}ki is supported by ANR project MAS (ANR-25-CE40-5294). $\ddagger$ Ramsey was supported by NSF CAREER award DMS-2442011}

\date{\today}

\begin{document}

\maketitle

\begin{abstract}We show that every faithful, transitive, and generically $(n+2)$-transitive action of a connected group $G$ on an irreducible variety $X$ of dimension $n > 0$, all defined over an algebraically closed field $F$, is isomorphic to the natural action of the projective linear group $PGL_{n+1}(F)$ on the projective space $\mathbb{P}^n(F)$. More precisely, we establish the Borovik--Cherlin conjecture for permutation groups $(G,X)$ definable in models of $ACF$.\end{abstract}

\section{Introduction}The main result of this paper yields the following characterisation. Let $\alpha: G\times X \rightarrow X$ be a faithful and transitive algebraic action of a connected algebraic group $G$ on an irreducible algebraic variety $X$ of dimension $n$, where both $G$ and $X$ are defined over an algebraically closed field $F$. If the induced diagonal action of $G$ on $X^{n+2}$ has a Zariski--open orbit, then the permutation group $(G,X)$ is isomorphic to $(PGL_{n+1}(F), \mathbb{P}^n(F))$ with the natural action of the projective linear group $PGL_{n+1}(F)$ on the projective space $\mathbb{P}^n(F)$.

Although the statement above is purely algebraic, and so are most of the techniques in this paper, our motivation comes from model theory. \emph{Groups of finite Morley rank} can be thought of as model theorists' approach to algebraic groups over algebraically closed fields. The former class is strictly broader than the latter, but the two are closely connected. This connection was highlighted by the famous \emph{Cherlin--Zilber conjecture} \cite{Cherlin1979, Zilber1977} proposing that infinite \emph{simple} groups of finite Morley rank are isomorphic to algebraic groups over algebraically closed fields. While the conjecture remains open, the past two decades have seen a number of fundamental results that provide strong evidence for its validity in the presence of involutions. One such result is the following theorem of Borovik and Cherlin \cite{BC08}: there exists a function $f:\mathbb{N} \to \mathbb{N}$ such that every definably primitive permutation group $(G,X)$ of finite Morley rank satisfies $$RM(G)\leq f(RM(X)).$$

In the same paper \cite{BC08}, Borovik and Cherlin introduced the notion of a \emph{generically $n$-transitive action}, extending the corresponding notion from algebraic group actions to the setting of finite Morley rank. A definable action of a group of finite Morley rank $G$ on a set $X$ of finite Morley rank is generically $n$-transitive if the induced diagonal action of $G$ on $X^n$ has an orbit $\mathcal{O}$ such that $$RM(X^n\setminus \mathcal{O}) < RM(X^n).$$ They proposed the following conjecture, asserting that generically highly transitive actions of finite Morley rank are rare.

\begin{conj}[The Borovik--Cherlin conjecture {\cite[Problem 9]{BC08}}]\label{conj:BC}Let $G$ be a connected group of finite Morley rank acting faithfully, definably, transitively and generically $(n+2)$-transitively on a set $X$ of Morley rank $n \geqslant 1$. Then the pair $(G, X)$ is equivalent to the projective linear group $PGL_{n+1}(F)$ acting on the projective space $\mathbb{P}^n(F)$ for some algebraically closed field $F$.\end{conj}

Substantial progress has recently been made towards Conjecture~\ref{conj:BC}. A notable breakthrough is due to Berkman and Borovik \cite{BB4}, who showed that the case where the underlying set $X$ is a connected abelian group cannot provide examples beyond the expected linear action. More precisely, they proved \cite[Theorem 3]{BB4} that if a connected group $G$ of finite Morley rank acts definably, faithfully, and generically $n$-transitively on a connected abelian group $X$ of finite Morley rank, where $n \geqslant RM(X)$, then $(G,X)$ is definably isomorphic to $(GL_n(F), F^n)$ for some algebraically closed field $F$. In particular, $RM(X)=n$.

It is natural to ask whether the Borovik--Cherlin conjecture admits an algebraic counterpart, namely, whether it holds when $G$ is a connected algebraic group and $X$ is an irreducible variety, both defined over an algebraically closed field. Until recently, this question remained open. In characteristic $0$, Popov \cite{Popov2007} studied generically highly transitive actions of simple groups (and reductive groups). He obtained a complete classification of such actions. Building on this classification, Freitag and Moosa \cite{Freitag-Moosa2025} recently established Conjecture~\ref{conj:BC} in the special case where the permutation group $(G,X)$ is definable in an algebraically closed field of characteristic $0$. The main result of the present paper removes this restriction on the characteristic:

\begin{customthm}{2.9(a)}Let $G$ be a connected group of finite Morley rank acting faithfully, definably, transitively and generically $(n+2)$-transitively on a set $X$ of Morley rank $n \geqslant 1$. Suppose that $(G,X)$ is definable in $F \models ACF$. Then $(G, X)$ is definably isomorphic to the natural action of the projective linear group $PGL_{n+1}(F)$ on the projective space $\mathbb{P}^n(F)$.
\end{customthm}

In a related recent paper \cite{FJM}, the authors proved Conjecture~\ref{conj:BC} in the special case where $(G,X)$ is definable in a differentially closed field of characteristic $0$. This is a setting in which the Cherlin--Zilber conjecture is known to hold. Their proof relies crucially on the result of Freitag and Moosa discussed above. Using the same philosophy, we show that the Cherlin--Zilber conjecture implies the Borovik--Cherlin conjecture (Theorem 2.9(b)).

\section{Generically $(n+2)$-transitive actions}We begin by fixing some notation and terminology. We use the \emph{Morley rank}, denoted by $RM$, and \emph{SU-rank} when it is relevant; no prior familiarity with these notions is assumed. We refer the reader to \cite{Borovik-Nesin,Wagner2000} for general background on groups of finite Morley rank and finite SU-rank, respectively. The notion of SU-rank will only be used in Proposition~\ref{lemma:parabolic-action}, where we compare it with algebraic dimension for Chevalley groups over pseudo-finite fields.

Throughout, Chevalley groups are considered primarily over algebraically closed fields. In one instance, namely Proposition~\ref{lemma:parabolic-action}, we also consider Chevalley groups over pseudo-finite fields. Recall that a Chevalley group is a group of the form $$G=X(F),$$ where $$X\in \{A_n (n \geqslant 1), B_n (n \geqslant 2), C_n (n \geqslant 3), D_n (n \geqslant 4), E_6, E_7, E_8, F_4, G_2\}.$$ The symbol $X$ denotes the type of the group, recording both its Lie type and Lie rank. A Chevalley group is called \emph{classical} if its type belongs to $\{A_n, B_n, C_n, D_n\}$.

We use the term \emph{simple group} in the sense of abstract group theory: a group with no non-trivial proper normal subgroups, including finite ones. Thus, when referring to a simple Chevalley group, we mean an abstractly simple group rather than a group which is simple modulo a finite centre. This distinction is only relevant in Proposition~\ref{lemma:parabolic-action}, where we consider simple Chevalley groups (which using our terminology are abstractly simple).

If $Y$ is an algebraic variety, we write $\dim(Y)$ for the dimension of $Y$ as an algebraic variety. 

The following definition, due to Popov \cite{Popov2007}, will be used throughout.

\begin{defn}\label{def:generictrans}An algebraic action $\alpha: G\times X \rightarrow X$ of a connected algebraic group $G$ on an irreducible algebraic variety $X$ is called \emph{generically $n$-transitive} if the induced diagonal action of $G$ on $X^n$ has a Zariski--open $G$-orbit. 

The \emph{generic transitivity degree} of an action $\beta$ is defined as $$
gtd(\beta):=\sup\{n:\beta\text{ is generically }n\text{-transitive}\}.
$$\end{defn}

The natural action of $GL_n(\mathbb{C})$ on $\mathbb{C}^n$ is a standard example of a generically $n$-transitive action, where the generic orbit consists of the set of bases for the vector space. Another example, central to this paper, is the natural action of $PGL_{n+1}(\mathbb{C})$ on $\mathbb{P}^n(\mathbb{C})$, which is generically $(n+2)$-transitive; the generic orbit is the set of projective bases.

As recalled in the introduction, Borovik and Cherlin \cite{BC08} extended Definition~\ref{def:generictrans} to actions of groups of finite Morley rank. Namely, a definable action of a group $G$ of finite Morley rank on a set $X$ of finite Morley rank is generically $n$-transitive if the diagonal action of $G$ on $X^n$ has an orbit $\mathcal{O}$ satisfying $$
RM(X^n\setminus\mathcal{O})<RM(X^n).
$$
For algebraic actions, this definition agrees with Definition~\ref{def:generictrans}; see \cite[Lemma 6.1]{Freitag-Moosa2025}. Thus the algebraic notion may equivalently be obtained by replacing Morley rank with algebraic dimension. Similarly, the definition extends naturally to definable actions of groups of finite SU-rank by replacing Morley rank with SU-rank.

\subsection*{Proofs}In Proposition~\ref{lemma:parabolic-action}, we consider Chevalley groups over both algebraically closed and pseudo-finite fields. The arguments in the two settings are essentially identical, and we include the pseudo-finite case since it may be useful for future work towards an analogue of Conjecture~\ref{conj:BC} for actions of finite SU-rank groups on finite SU-rank sets definable in $ACFA$ (that is, when one aims to classify the highly generically transitive actions definable in difference algebraic fields). In the pseudo-finite case, we will use the following fact.

\begin{fact}[{\cite[Lemma 4.8]{Karhumaki-Ramsey2026}}] \label{lem: dimension inequality}
Suppose $G$ is a simple Chevalley group over a pseudo-finite field $F$. Then we have the following:
\begin{enumerate}
    \item There is a $\mathcal{L}_{gp}$-definable maximal torus $T \leq G$ with $SU(T) = r \cdot SU(F)$, where $r = dim(T)$ is the Lie rank of $G$.
    \item $SU(G) = dim(G) \cdot SU(F)$.
    \item If $P \leq G$ is a parabolic subgroup, then $P$ is $\mathcal{L}_{gp}$-definable and $SU(P) = dim(P) \cdot SU(F)$. 
\end{enumerate}  
\end{fact}

\begin{prop}\label{lemma:parabolic-action}Let $G$ be a simple Chevalley group of type different from $A_n$ over a field $F$ which is either algebraically closed or pseudo-finite. Suppose that $P \leqslant G$ is a parabolic subgroup. Let $rk$ stand for $dim$ when $F$ is algebraically closed and for $SU$ when $F$ is pseudo-finite, and suppose that $rk(G/P)=k$. Then the permutation group $(G,G/P)$ is not generically $(k+2)$-transitive for any $k \geqslant 1$.\end{prop}

\begin{proof}Towards a contradiction, suppose that the action is generically $(k+2)$-transitive. If $F$ is algebraically closed this entails $dim(G )\geqslant k(k+2)$ and if $F$ is pseudo-finite then, together with Fact~\ref{lem: dimension inequality}, this entails $$dim(G) \geqslant  dim(G/P)(SU(G/P)+2),$$ in particular, $$dim(G) \geqslant dim(G/P)(dim(G/P)+2).$$ So setting $\ell=dim(G/P)$ we have $dim(G)\geqslant \ell(\ell+2)$.

Both $dim(G)$ and $dim(P)$ can be read off from their associated Dynkin diagrams, since the algebraic dimensions of $G$ and $P$ correspond to numbers of positive roots in the root systems, which are, in turn, determined by the diagrams. Indeed, since $G=X(F)$ for $X\in \{B_n,C_n,D_n,E_6,E_7,E_8,F_4,G_2\}$ we have $$dim(G)=n+2 \cdot |\Phi^+|$$ (for exceptional groups $n\in \{6,7,8,4,2\}$) and $$dim(P)=n+  |\Phi^+|+ |\Phi_I^+|,$$ where $|\Phi^+|,|\Phi_I^+|$ are respectively the number of positive roots in the Dynkin diagrams of $G$ and $P$. So $$\ell=|\Phi^+|-|\Phi_I^+|.$$ 

We will show that none of the cases $X\in \{B_n,C_n,D_n,E_6,E_7,E_8,F_4,G_2\}$ can happen, using simple calculations.

If $G$ is of exceptional type $E_6,E_7, E_8, F_4,$ or $G_2$ then the dimension of $G$ is equal to $78,133, 248,52,14$, respectively. So, from $dim(G)\geqslant \ell(\ell+2)$, we have $$\ell \leqslant 7,10,14,6,2,$$ respectively. 
But even when $P$ is a maximal parabolic (that is, $\ell$ is as small as possible) this cannot happen: this can be checked with a straightforward calculation as, up to conjugacy, maximal parabolic subgroups of $G$ are determined by deleting a single vertex from the Dynkin diagram of $G$, and it is known that $|\Phi^+|$ is equal to $36, 63, 120, 24, 6$ for the types $E_6,E_7, E_8, F_4, G_2$ respectively and equal to $ \frac{n^2+n}{2},n^2,n^2, n(n-1)$ for the types $A_n(n\geqslant 1), B_n(n\geqslant 2), C_n(n\geqslant 3), D_n(n\geqslant 4)$ respectively.

The Dynkin diagrams of Chevalley groups are drawn in \cite[p. 40]{Carter1971}. We now eliminate the exceptional cases using $\ell=|\Phi^+|-|\Phi_I^+|$ (below, when deleting two different nodes from a Dynking diagram produce products of the same diagrams but in different orders, e.g. $A_2 \times A_1$ and $A_1 \times A_2$, we do not mention both as the corresponding maximal parabolic subgroups have the same number of positive roots):

\begin{itemize}
\item $E_6$: we have $|\Phi^+|=36$ and a maximal parabolic $P$ has a Dynkin diagram of the form $A_5,D_5, A_1 \times A_4$ or $A_1 \times A_2\times A_2$. So $|\Phi_I^+|$ is equal to $15,20,11$ or $7$. Thus $\ell \geqslant 16 > 7$.
\item $E_7$: we have $|\Phi^+|=63$ and a maximal parabolic $P$ has a Dynkin diagram of the form $A_6,E_6,A_1 \times D_5, A_2 \times A_4, A_1 \times A_2\times A_3, A_5 \times A_1$ or $D_6$. So $|\Phi_I^+|$ is equal to $21,36,21,13,10,16$ or $30$. Thus $\ell \geqslant 27 > 10$.
\item $E_8$: we have $|\Phi^+|=120$ and a maximal parabolic $P$ has a Dynkin diagram of the form $A_7, E_7, A_1\times E_6, A_2\times D_5, A_3 \times A_4,A_4\times A_2 \times A_1, A_6\times A_1$ or $D_7$. So $|\Phi_I^+|$ is equal to $28,63,37,23,16,14,22$ or $42$. Thus $\ell \geqslant 57 > 14$.
\item $F_4$: we have $|\Phi^+|=24$ and a maximal parabolic $P$ has a Dynkin diagram of the form $C_3, B_3$ or $A_1 \times A_2$. So $|\Phi_I^+|$ is equal to $9$ or $4$. Thus $\ell\geqslant 15 > 6$.
\item $G_2$: we have $|\Phi^+|=6$ and a maximal parabolic $P$ has a Dynkin diagram of type $A_1$ so $|\Phi_I^+|=1$. Thus $\ell=5 > 2$.
\end{itemize}

From now on, we may assume that $G$ is of classical type. For types $B_n$ and $C_n$ we have $dim(G)=n+2n^2$ and for the type $D_n$ we have $dim(G)=2n^2-n$. Thus, since $dim(G) \geqslant \ell(\ell+2)$, in the case $X\in \{B_n(n\geqslant 2) ,C_n(n\geqslant 3)\}$ it is enough to show that $\ell \geqslant 2n-1$. In the case $X=D_n(n\geqslant 4)$, it is enough to show that $\ell \geqslant 2n-2$. It then follows that for types $B_n,C_n$ we have $$2n^2 +n=dim (G) \geqslant \ell(\ell+2)\geqslant 4n^{2} - 1,$$ which is not true for any $n\geqslant 2$. Similarly for type $D_n$ we have $$2n^2 -n=dim (G)\geqslant  \ell(\ell+2) \geqslant 4n^2-4n,$$ which is not true for any $n\geqslant 4$. 

Since $\ell$ is as small as possible when $P$ is a maximal parabolic, we may assume that $P$ is maximal. Below we consider the types $B_n,C_n$ and $D_n$ separately, again using $\ell=|\Phi^+|-|\Phi_I^+|$.

\,

    \textbf{Case 1}:  $G$ has Dynkin diagram $B_{n}(n\geqslant 2)$ or $C_{n}(n\geqslant3)$. We show that $\ell \geqslant 2n-1$.

    Recall that a root system with Dynkin diagram $B_{r}$ or $C_{r}$ has $r^{2}$ positive roots. 
    
    Removing a vertex from a diagram of type $B_{n}$ can result either in

     \begin{enumerate}[(i)]
    \item a diagram of type $B_{n-1}$,
    \item a diagram of type $A_{i} \times B_{j}$ for $i+j=n-1$, where $j \geqslant 2$,
    \item a diagram of type $A_{n-2} \times A_{1}$, or 
    \item a diagram of type $A_{n-1}$. 
\end{enumerate}

    In case (i) we have $$\ell=n^2-(n^2-2n+1)=2n-1\geqslant 2n-1.$$ 
    
Since a root system with Dynkin diagram of type $A_{n}$ has $\frac{n^2+n}{2}$ positive roots, in case (iv) we have for all $n \geqslant 2$ that $$\ell=n^2-\frac{(n-1)^2+(n-1)}{2}=n^2-\frac{(n-1)n}{2}=\frac{n^2}{2}+\frac{n}{2}\geqslant 2n-1.$$

    In case (ii) we have
    $$
 \ell=  n^{2} - \left(\frac{i(i+1)}{2} + j^{2} \right)
    $$
    when $0 \leq i,j$ and $i+j =n-1$ and $j\geqslant 2$. If $i = 0$ we are in case (i). So we may suppose that $i,j > 0$, and thus for all $n \geqslant 2$ we have
    \begin{eqnarray*}
    n^{2} - \frac{i(i+1)}{2} - j^{2} &=& n^{2} - (i+j)^{2} +2ij + \frac{i^{2}}{2} - \frac{i}{2} \\
    &=& n^{2} - (n-1)^{2}  +2ij + \frac{i^{2}}{2} - \frac{i}{2}\\
    &=& 2n-1  +2ij + \frac{i^{2}}{2} - \frac{i}{2}  \\
    &\geq& 2n-1.
    \end{eqnarray*}

    For (iii), we need to check that for all $n \geqslant 2$ we have
    $$n^{2} - \left( \frac{(n-2)^2 + (n-2)}{2} + \frac{1^2+1}{2} \right)=
    n^{2} - \frac{(n-2)(n-1)}{2} - \frac{2}{2} \geq 2n-1.
    $$
    
  As the left hand side is equal to $  n^{2} -\frac{n^2}{2}+ \frac{3n}{2}-
  2$, by easy algebra, one checks that this is true. 
  
  The argument is identical for Dynkin diagrams of type $C_{n}$. 

\,

    \textbf{Case 2}:  $G$ has Dynkin diagram $D_{n}(n \geqslant 4)$. We show that $\ell \geqslant 2n -2$.

    A root system of type $D_{r}$ has $r(r-1)$ positive roots. Removing a vertex from a Dynkin diagram of type $D_{n}$ results in either 
     \begin{enumerate}[(i)]
    \item a diagram of type $D_{n-1}$,
    \item a diagram of type $A_{i} \times D_{j}$ for $i+j = n-1$, $j\geqslant 4$,
    \item a diagram of type $A_{n-3} \times A_{1} \times A_{1}$, 
    \item a diagram of type $A_{n-1}$, or
    \item a diagram of type $A_{n-4} \times A_{3}$.
\end{enumerate}

    In case (i) we have $$\ell=n(n-1)-((n-1)(n-2))=n^2-n-n^2+3n-2=2n-2\geqslant 2n-2.$$

    In case (iv) we have for all $n \geqslant 4$ that
    
    $$\ell=n(n-1)-\frac{(n-1)^2+(n-1)}{2}=n^2-n-\frac{(n-1)n}{2}\geqslant 2n-2.$$

By the above, in case (ii) we may assume that $i, j > 0$. We then need to verify
    $$
    \ell=n(n-1) - \frac{i(i+1)}{2} - j(j-1) \geq 2n-2. 
    $$
    Expanding, we have 
    \begin{eqnarray*}
        n(n-1) - \frac{i(i+1)}{2} - j(j-1) &=& n^{2} - n - \frac{i^{2}}{2} - \frac{i}{2} - j^{2} + j \\
        &=& n^{2} - n - (i+j)^{2} + \frac{i^{2}}{2} + 2ij - \frac{i}{2}  + j \\
        &=& n^{2} - n - (n-1)^{2} + \frac{i^{2}}{2} + 2ij - \frac{i}{2}  + j \\
        &=& n - 1 + \frac{i^{2}}{2} +  2ij - \frac{i}{2} + j,
    \end{eqnarray*}
    so we need to establish
    $$
    \frac{i^{2}}{2} - \frac{i}{2}  + 2ij + j   \geq n-1=i+j.
    $$
    
  This simplifies to $$ \frac{i^{2}}{2} - \frac{i}{2}  + 2ij \geq i$$ which is clearly true by the choice of $i,j$.

    For case (iii), we have to show that 
    $$
    n(n-1) - \frac{(n-3)(n-2)}{2} - \frac{1(2)}{2} - \frac{1(2)}{2} \geq 2n-2.
    $$
    Simplifying, this is equivalent to showing 
    $$
    n^{2} - n - 6 \geq 0,
    $$
    and this is true for all $n \geq 4$.
    
    Finally, we consider case (v). Note that this can only happen for $n \geqslant 5$. We now calculate 
    \begin{eqnarray*}
        \ell &=& n(n-1) - \frac{(n-4)^{2} + (n-4)}{2} - \frac{3^{2}+3}{2} \\
        &=& \frac{2n^{2} - 2n - n^{2} + 8n - 16 - n + 4}{2} - 6 \\
        &=& \frac{n^{2} +5n -12}{2} - 6.
    \end{eqnarray*}
    At $n = 5$ this evaluates to $13$ which is greater than $2n - 2 = 8$ and it grows faster than $2n - 2$, establishing the desired inequality.  \end{proof}

    The following observation is made in \cite[Lemma 2(i)]{Popov2007} but we re-observe it here as we avoid referring to \cite{Popov2007} where fields are of characteristic 0. 

\begin{fact}\label{fact:gct}Let $\alpha_i$ be an action of a connected algebraic group $G$ on an irreducible algebraic variety $X_i$, $i \in\{1,2\}$. Assume that there exists a $G$-equivariant dominant rational map $\phi: X_1 \dashrightarrow X_2$. Then $gtd(\alpha_1) \leqslant gtd(\alpha_2)$.
\end{fact}
\begin{proof}Suppose that the action $\alpha_1$ is generically $n$-transitive. Since the rational map $\phi_n: X_1^n \dashrightarrow X_2^n$ defined by $(x_1, \ldots, x_n)\mapsto (\phi(x_1), \ldots, \phi(x_n))$ is $G$-equivariant and dominant, the indeterminacy locus of $\phi_n$ lies in the complement to the open $G$-orbit in $X_1^n$, and the image of this orbit under $\phi_n$ is a $G$-orbit open in $X_2^n$.\end{proof}

Recall that a \emph{reductive} algebraic group is an algebraic group with trivial unipotent radical. Let $H$ and $K$ be closed subgroups of a reductive algebraic group $G$, defined over an algebraically closed field. Then, $H \times K$ acts on $G$ by $(x,y) \cdot g = xgy^{-1}$ for $(x,y)\in H \times K$, $g\in G$, and the orbits are the \emph{$H,H$-double cosets} in $G$.
    
    \begin{fact} \cite[Theorem 1.6]{brundan2000double} \label{fact: Brundan}
Let $H$ be a proper reductive subgroup of a connected reductive algebraic group $G$. Then there is no dense $H$,$H$-double coset in $G$.  
\end{fact}

\begin{lem} \label{lem: double coset lemma}Suppose $G$ is a connected reductive algebraic group and $H$ is a proper reductive subgroup of $G$. Then there is no orbit $\mathcal{O}$ in the action of $G$ on $(G/H) \times (G/H)$ with $dim(\mathcal{O}) = 2 dim (G/H)$. In particular, the generic transitivity degree of the action of $G$ on $G/H$ is equal to $1$. 
\end{lem}

\begin{proof}
    The generic transitivity degree of the action is clearly at least $1$ since $G$ acts transitively on $G/H$. Set $X:=(G/H) \times (G/H)$ and suppose, towards a contradiction, that there is an orbit $\mathcal{O}$ with $dim(\mathcal{O}) = 2 dim(X)$.  Then
    $$
   dim(\mathcal{O}) = 2 dim(G/H) = 2 dim(G) - 2 dim(H). 
    $$
    Note that the stabiliser of $(H,gH)$ in $G$ is $H \cap gHg^{-1}$, thus we have 
    $$
    dim(G) - dim(H \cap gHg^{-1}) = 2 dim(G) - 2 dim(H),
    $$
    which rearranges to give 
    $$
    dim(G) = 2 dim(H) - dim(H \cap gHg^{-1}). 
    $$
    Now we calculate the dimension of the double coset $dim(HgH)$. The group $H \times H$ acts transitively on $HgH$ via $(a,b) \cdot hgh' = ahgh'b^{-1}$ for all $a,b,h,h' \in H$. The stabiliser of $g$ under this action is the set of pairs $(h,g^{-1}hg)$ which lie in $H^{2}$. The projection to the left-coordinate is a bijection with $H \cap gHg^{-1}$, so we have 
    $$
   dim(HgH) = 2 dim(H) - dim(H \cap gHg^{-1}). 
    $$
    This yields the equality 
    $$
    dim(G) = dim(HgH),
    $$
    and, since $G$ was assumed connected, this entails that $HgH$ is a dense $H$,$H$-double coset, contradicting Fact \ref{fact: Brundan}. 
\end{proof}

\begin{fact} \cite[Theorem 30.4(a)]{humphreys1975} \label{fact: parabolic or reductive}
    Let $G$ be a reductive algebraic group, and $H \leq G$ a maximal proper closed subgroup. Then either $H^{\circ}$ is reductive, or $H$ is parabolic. 
\end{fact}

\begin{lem}\label{lem:parabolic}
    Suppose $G$ is a connected reductive algebraic group and $H \leq G$ is a closed subgroup. If the action of $G$ on $G/H$ is generically $2$-transitive, then $H$ is contained in a parabolic subgroup. 
\end{lem}

\begin{proof}
    Suppose not, so we assume towards contradiction that $H$ is not contained in a parabolic subgroup of $G$. By Fact~\ref{fact:gct}, we may assume $H$ is maximal and thus, by Fact \ref{fact: parabolic or reductive}, we must have that $H^{\circ}$ is reductive. If there is a generic orbit $\mathcal{O}$ in the action of $G$ on $(G/H)^{2}$, then $\mathcal{O}$ can be written as a union of at most finitely many orbits $\mathcal{O}_{1}, \ldots, \mathcal{O}_{k}$ in $(G/H^{\circ})^{2}$. Then there must be some $i$ with 
    $$
   dim(\mathcal{O}_{i}) = dim(\mathcal{O}) = 2 dim(G/H) = 2 dim(G/H^{\circ}),
    $$
    which contradicts Lemma \ref{lem: double coset lemma}. 
\end{proof}

\begin{thm}\label{th:main}Let $G$ be a connected group of finite Morley rank acting faithfully, definably, transitively and generically $(n+2)$-transitively on a set $X$ of Morley rank $n \geqslant 1$. Suppose that one of the following happens:
\begin{enumerate}[(a)]
\item  $(G,X)$ is definable in $F \models ACF$.
\item The Cherlin--Zilber conjecture holds.
\end{enumerate}Then $(G, X)$ is definably isomorphic to the natural action of the projective linear group $PGL_{n+1}(F)$ on the projective space $\mathbb{P}^n(F)$.\end{thm}

     \begin{proof}
We prove the theorem by induction on $n$. Note first that connectedness of $G$ and the transitivity of the action entail that $X$ is of Morley degree $1$ (\cite[Lemma 1.4]{BC08}). This means that for $n=1$ the set $X$ is strongly minimal. So, as generic $3$-transitivity implies $RM(X) \geqslant 3$, by Hrushovski's famous classification of definable actions on strongly minimal sets (see e.g. \cite[Theorem 11.98]{Borovik-Nesin}), there is an algebraically closed field $F$ such that the pair $(G,X)$ is definably isomorphic to $(PGL_2(F),\mathbb{P}^1(F))$. 

So suppose $n > 1$. Set $G_x$ to be a point-stabiliser of $x\in X$. Consider a proper definable subgroup $H$ of $G$ containing $G_x$. We show that $H$ is a finite extension of $G_x$. The action of $G$ on $G/H$ is still transitive and generically $(n+2)$-transitive (Fact~\ref{fact:gct}), but it may not be faithful. Taking $R$ to be the kernel of this action, the connected group $G/R$ acts on $G/H$ is faithfully, transitively and generically $(n+2)$-transitively. If the index $[H:G_x] $ was infinite then $RM(G/H):= e$ is strictly smaller than $n > 1$. Then $(G/R, G/H)$ is definably isomorphic to $(PSL_{e+1}(F), \mathbb{P}^{e}(F))$. This means that the action is not generically $(e+3)$-transitive for $e < n$ contradicting generic $(n+2)$-transitivity. So $H$ is a finite extension of $G_x$.

Note then that by transitivity we may recognise $X$ as $G/G_x$ so $G_x$ is infinite as generic $(n+2)$-transitivity entails $RM(G) \geqslant n(n+2)$. This means that the connected component $G_x^\circ \unlhd G_x$ is infinite. Now set $L=N_G(G_x^\circ)$ and note that $L$ is a proper subgroup of $G$, for otherwise the infinite group $G_x^\circ$ would stabilise each point of $G/G_x$ contradicting the fact that $G$ acts on $G/G_x$ faithfully. By the above $L$ is a finite extension of $G_x$. The above also implies that $L$ is a maximal definable subgroup of $G$ as, for any proper definable subgroup $H$ of $G$ containing $G_x$ we have $G_x^\circ=H^\circ$, so $H \leqslant L$. So, if we prove that $G/L=\mathbb{P}^n(F)$ then $L$ is connected and thus equal to $G_x$. Let $K$ be the kernel of the action of $G$ on $G/L$ and set $\overline{G}:=G/K$, $\overline{L}:=L/K$. 

We claim that, in order to prove the theorem, it suffices to show that $(\overline{G}, \overline{G}/\overline{L})$ is definably isomorphic to $(PGL_{n+1}(F), \mathbb{P}^{n}(F))$. To see this, note that $$K = \bigcap_{g \in G} N_{G}(G^{\circ}_{x})^{g} = \bigcap_{g \in G} N_{G}(G^{\circ}_{g\cdot x}).$$ Then $K$ is finite, since $K^{\circ}$ is contained each conjugate of $G^{\circ}_{x}$ and is therefore trivial. Also, since $G$ is connected and $K$ is normal, we get that $K$ is central. The parabolic subgroup $\overline{L}$ of $\overline{G}$ is connected \cite[Corollary 6.4.10]{springer1998linear}, so we get $L = L^{\circ} K$. If $k \in G_{x} \cap K$ and $k \neq 1$, then, by transitivity and faithfulness of the action of $G$ on $X$, there is some $g \in G$ such that $kg \cdot x \neq g\cdot x = gk \cdot x$. Then $kG^{\circ}_{g\cdot x} k^{-1} \neq G^{\circ}_{g \cdot x}$, contradicting $k \in N_{G}(G^{\circ}_{g \cdot x})$ for all $g$. This shows $G_{x} \cap K = 1$. Moreover, we have $L^{\circ} = G^{\circ}_{x}$ so we have $L = L^{\circ} \times K$ and $G_{x} = L^{\circ}$ (so, in particular, $G_{x} = G^{\circ}_{x}$). We also know $G$ is perfect since $\overline{G}$ is perfect and the derived subgroup $G'$ is definable so $G = G'K$ and thus the connectedness of $G$ entails $G = G'$. We have thus shown that $G$ is a perfect central extension of $\overline{G}$. Thus, under the assumption that $(\overline{G},\overline{G}/\overline{L})$ is definably isomorphic to $(PGL_{n+1}(F), \mathbb{P}^{n}(F))$, we see, by \cite[Corollary 1]{altinel1999central}, that $G$ is necessarily algebraic and $L$ is a parabolic subgroup of $G$, and is therefore connected as well \cite[Corollary 6.4.10]{springer1998linear}. But then the identity $L = L^{\circ} \times K$ forces $K = 1$ and we obtain $(G,G/L) \cong (\overline{G},\overline{G}/L) \cong (PGL_{n+1}(F),\mathbb{P}^{n}(F))$. 

Thus, we prove the result for $(\overline{G}, \overline{G}/\overline{L})$; note again that this is a permutation group of finite Morley rank where the action is transitive and generically $(n+2)$-transitive. 

Clearly $\overline{L}$ is a maximal definable subgroup of $\overline{G}$, so the permutation group $(\overline{G}, \overline{G}/\overline{L})$ is definably primitive. Now we may apply `the finite Morley rank analogue of the O'Nan-Scott theorem' by Macpherson and Pillay \cite[Theorem 1.1]{Macpherson-Pillay1995}. They show that either $(\overline{G}, \overline{G}/\overline{L})$ is of `affine type', or there is a definable subgroup $\overline{S}\unlhd \overline{G}$ (called the definable socle) so that one of the following holds (the list in \cite[Theorem 1.1]{Macpherson-Pillay1995} is longer but as they point out connectedness of $\overline{G}$ gives us the reduced list here): \begin{enumerate}[(i)]
    \item $\overline{S}$ is simple non-abelian and $\overline{S} \unlhd \overline{G} \leqslant Aut(\overline{S})$, or
    \item $\overline{S}=\overline{T}_1 \times \overline{T}_2,$ where $\overline{T}_1,\overline{T}_2$ are (definably isomorphic) infinite simple non-abelian definable normal subgroups of $\overline{G}$ both acting regularly on $\overline{G}/\overline{L}$, and $\overline{S}\unlhd \overline{G} \leqslant \overline{W}$ where $\overline{W}$ is an extension of $\overline{S}$ by $ Aut(\overline{T}_1)/\overline{T}_1 \times Sym_2.$
\end{enumerate}

The recent result \cite[Theorem 2]{BB4} immediately implies that $(\overline{G}, \overline{G}/\overline{L})$ cannot be of affine type, thus one of the cases (i),(ii) happens. Now we start using our assumptions (a) or (b). In case (a), it follows from well-known results of Weil, van den Dries and Hrushovski (see \cite[Theorem 1]{Bouscaren1989}) that $\overline{G}$ is definably isomorphic to an algebraic group $G(K)$ over an algebraically closed field $K$ and $\overline{G}/\overline{L}$ is definably isomorphic to an irreducible variety $V(K)$ over $K$. Thus the definable subgroups $\overline{S}$ in (i) and $\overline{T}_1, \overline{T}_2$ in (ii) are simple algebraic groups. Similarly, the assumption (b) implies that the definable subgroups $\overline{S}$ in (i) and $\overline{T}_1, \overline{T}_2$ in (ii) are simple algebraic groups. As the groups of definable field automorphisms of the simple algebraic groups $\overline{S}$, $\overline{T}_1$ are trivial \cite[Lemma 4.1]{Altinel-Borovik-Cherlin2008}, the groups of definable outer automorphisms $DefAut(\overline{S})/\overline{S}$, $DefAut(\overline{T}_1)/\overline{T}_1$ are finite. So $\overline{G}$ is a finite extension of $\overline{S}$, thus equal to $\overline{S}$ by connectedness. At the same time generic $(n+2)$-transitivity entails that $RM(\overline{G})> 2 \cdot RM(\overline{G}/\overline{L})$ so case (ii) cannot happen.

Now the simple non-abelian algebraic group $\overline{G}$ is reductive so by Lemma~\ref{lem:parabolic} $\overline{L}$ is a parabolic subgroup. Lemma~\ref{lemma:parabolic-action} then gives us that $\overline{G}$ is of type $A_d$, that is, $\overline{G}$ is definably isomorphic to $PGL_{d+1}(F)$. As $\overline{L}$ is a parabolic subgroup, we have that $n = RM(\overline{G}/\overline{L}) \geqslant dim(\overline{G}/\overline{L})\geqslant dim(\overline{H})$ (\cite[Fact 4.9]{Karhumaki-Ramsey2026}), where $\overline{H}$ is a maximal torus of $\overline{G}=PGL_{d+1}(F)$. Since the dimension of a maximal torus equals the Lie rank of the group $PGL_{d+1}(F)$, we get that $n \geqslant d$. This means that $dim(\overline{G}/\overline{L})$ is at most $d$: As $\overline{G}$ acts generically $(n+2)$-transitively on $\overline{G}/\overline{L}$ we have that $d^2+2d=dim(G)\geqslant dim(\overline{G}/\overline{L}) (n+2)$. So if $dim(\overline{G}/\overline{L}) $ was strictly bigger than $d$ then we would get that $d > n$, which is not the case. At the same time, since $$dim(\overline{G}/\overline{L})+dim(\overline{L})=dim(\overline{G})=d^2+2d,$$ if $dim(\overline{G}/\overline{L}) < d$  then $dim(\overline{L}) > d^2+d$. But any maximal parabolic of $\overline{G}=PGL_{d+1}(F)$ is of dimension at most $d^2+d$. This means that $dim(\overline{G}/\overline{L})$ is at most $d$ and hence we have $dim(\overline{G}/\overline{L}) = d$. We may now conclude that $\overline{G}/\overline{L} \cong \mathbb{P}^{d}(F)$ as follows. The maximal parabolic $\overline{L}$ is conjugate to a stardard parabolic of $\overline{G}=PGL_{d+1}(F)$ so $\overline{L}$ is the stabiliser of the standard $k$-dimensional subspace of $F
^{d+1}$ for $1 \leqslant k \leqslant d$. This means that $\overline{G}/\overline{L}$ is the Grassmannian $GR(k, d+1)$ so $d=dim(\overline{G}/\overline{L})=k(d+1-k)$. That is, either $k=1$ or $k=d$. If $k=1$ then we have $GR(1, d+1) \cong  \mathbb{P}^{d}(F)$. If $k=d$ we have $\overline{G}/\overline{L} \cong GR(d, d+1)$, which is naturally isomorphic to $\mathbb{P}^{d}(F)$ by sending a hyperplane in $F^{d-1}$ to its corresponding point in the dual projective space. So we have identified $(\overline{G},\overline{G}/\overline{L})$ with $(PGL_{d+1}(F),{\mathbb{P}}^{d}(F))$, where $n \geqslant d$. But this means that $(\overline{G},\overline{G}/\overline{L})$ is not generically $k$-transitive for any $k > d+2$; so $n \leqslant d$. So we have $n=d$ and this finishes our proof.\end{proof}

We finish the paper with an easy observations.

\begin{cor}\label{corol:bounding-gct}Suppose $G$ is an algebraic group acting generically $k$-transitively on an infinite algebraic variety $X$. Then $k \leqslant dim(X)+2$. 
\end{cor}
\begin{proof}Since we may assume $k > 1$ the action of $G^\circ$ on $X$ is also generically $k$-transitive (\cite[Lemma 1.8]{BC08}). We may then replace $X$ by a generic orbit and, similarly as in the proof of Theorem~\ref{th:main}, we observe that $G/K$ acts generically $k$-transitively, transitively, and faithfully on $X$ where $K$ is the kernel of the action. So we replace $G$ by $G/K$. 
Set $n=dim(X)$. If $k \geqslant n+2$, then Theorem~\ref{th:main} implies that $(G,X)$ can be recognised as $(PGL_{n+1}(F),{\mathbb{P}}^{n}(F))$. But this means that $(G,X)$ is not generically $k$-transitive for any $k > n+2$; so $k \geqslant n+2$ implies $k= n+2$ as we wanted.
\end{proof}

\subsection*{Acknowledgements}We would like to thank James Freitag and R\'{e}mi Jaoui for stimulating conversations which ultimately led us to consider this problem.

\bibliographystyle{plain}
\bibliography{Ulla}{}

\begin{thebibliography}{10}

\bibitem{Altinel-Borovik-Cherlin2008}
Tuna Alt{\i}nel, Alexandre~V. Borovik, and Gregory Cherlin.
\newblock {\em Simple {G}roups of {F}inite {M}orley rank}.
\newblock American Mathematical Society Providence, Providence, RI, 2008.

\bibitem{altinel1999central}
Tuna Altinel and Gregory Cherlin.
\newblock On central extensions of algebraic groups.
\newblock {\em The Journal of Symbolic Logic}, 64(1):68--74, 1999.

\bibitem{BB4}
Ay{\c{s}}e Berkman and Alexandre~V. Borovik.
\newblock Primitive permutation groups of finite {M}orley rank and affine type.
\newblock {\em arXiv:2405.07307}, 2024.

\bibitem{BC08}
Alexandre Borovik and Gregory Cherlin.
\newblock Permutation groups of finite {Morley} rank.
\newblock In Z.~Chatzidakis, H.D. Macpherson, A.~Pillay, and A.J. Wilkie,
  editors, {\em Model Theory with applications to algebra and analysis, {I} and
  {II}}. Cambridge University Press, 2008.

\bibitem{Borovik-Nesin}
Alexandre~V. Borovik and Ali Nesin.
\newblock {\em Groups of {F}inite {M}orley {R}ank}, volume~26 of {\em Oxford
  Logic Guides}.
\newblock Oxford University Press, New York, 1994.
\newblock Oxford Science Publications.

\bibitem{Bouscaren1989}
Elisabeth Bouscaren.
\newblock Model theoretic versions of {W}eil's theorem on pregroups.
\newblock In A.~Nesin and A.~Pillay, editors, {\em Notre Dame Mathematical
  Lectures}, pages 177--185, Notre Dame, Indiana, 1989. Notre Dame University
  Press.

\bibitem{brundan2000double}
Jonathan Brundan.
\newblock Double coset density in classical algebraic groups.
\newblock {\em Transactions of the American Mathematical Society},
  352(3):1405--1436, 2000.

\bibitem{Carter1971}
Roger~W. Carter.
\newblock {\em Simple {G}roups of {L}ie {T}ype}.
\newblock Wiley, New York, 1971.

\bibitem{Cherlin1979}
Gregory Cherlin.
\newblock Groups of small {M}orley rank.
\newblock {\em Annals of Mathematical Logic}, 17(1):1--28, 1979.

\bibitem{FJM}
James Freitag, Léo Jimenez, and Rahim Moosa.
\newblock Finite-dimensional differential-algebraic permutation groups.
\newblock {\em Journal of the Institute of Mathematics of Jussieu},
  24(2):603–626, 2025.

\bibitem{Freitag-Moosa2025}
James Freitag and Rahim Moosa.
\newblock Bounding nonminimality and a conjecture of {B}orovik–{C}herlin.
\newblock {\em J. Eur. Math. Soc.}, 27(2):589–613, 2025.

\bibitem{humphreys1975}
{J}ames~{E}. Humphreys.
\newblock {\em Linear {A}lgebraic {G}roups}.
\newblock Graduate {T}exts in {M}athematics. Springer, 1975.

\bibitem{Karhumaki-Ramsey2026}
Ulla Karhum\"{a}ki and Nicholas Ramsey.
\newblock Primitive pseudo-finite permutation groups of finite {S}{U}-rank.
\newblock {\em To apper in the Transactions of the American Mathematical
  Society}.

\bibitem{Macpherson-Pillay1995}
Dugald Macpherson and Anand Pillay.
\newblock Primitive permutation groups of finite {M}orley rank.
\newblock {\em Proceedings of the London Mathematical Society},
  s3-70(3):481--504, 1995.

\bibitem{Popov2007}
Vladimir~L. Popov.
\newblock Generically multiple transitive algebraic group actions.
\newblock In {\em Algebraic groups and homogeneous spaces}, page 481–523.
  Tata Inst. Fund. Res. Stud. Math., Mumbai, 2007.

\bibitem{springer1998linear}
T.~A. Springer.
\newblock {\em Linear Algebraic Groups}.
\newblock Modern Birkh{\"a}user Classics. Birkh{\"a}user, Boston, MA, 2
  edition, 1998.
\newblock 2nd printing 2008.

\bibitem{Wagner2000}
Frank~O. Wagner.
\newblock {\em Simple {T}heories}.
\newblock Kluwer {A}cademic {P}ublishers, {D}ordrecht, {N}{L}, 2000.

\bibitem{Zilber1977}
Boris~I. Zilber.
\newblock Groups and rings whose theory is categorical.
\newblock {\em Fundamenta Mathematicae}, 95(3):173--188, 1977.

\end{thebibliography}

\end{document}